\documentclass[10pt]{article}
\usepackage{BICA,amssymb,amsmath,amsthm}
\newcommand{\ints}{\mathbb{Z}}

\newcommand{\cx}{\mathbb{C}}

\newcommand{\cD}{\mathcal{D}}

\newtheorem{thm}{Theorem}[section]
\newtheorem{lem}[thm]{Lemma}

\newtheorem{example}{Example}[section]
\newcommand{\zed}{\ensuremath{\mathbb{Z}}}
\DeclareSymbolFont{bbold}{U}{bbold}{m}{n}
\DeclareSymbolFontAlphabet{\mathbbold}{bbold}
\newcommand{\ones}{\mathbbold{1}}

\title{Some Nonexistence Results for Strong External Difference Families Using Character Theory}
\author{
 William~J.~Martin\\[2ex]
\it Department of Mathematical Sciences\\
\it Worcester Polytechnic Institute\\
\it 100 Institute Road, 
Worcester, MA 01609,
USA \\
\texttt{martin@wpi.edu}\\[2ex]
\and
Douglas~R.~Stinson\\[2ex]
\it David R.\ Cheriton School of Computer Science\\
\it  University of Waterloo\\
\it Waterloo, Ontario, N2L 3G1, Canada \\
 \texttt{dstinson@uwaterloo.ca}
}
\date{}

\begin{document}
\maketitle

\begin{abstract}
In this paper, we study the existence of $(v,m,k,\lambda)$-strong external difference 
families (SEDFs). We use character-theoretic techniques to show that no  SEDF exists
when $v$ is prime, $k > 1$ and $m > 2$. In the case where $v$ is the product of two 
distinct odd primes, some necessary conditions are derived, which can be used to 
rule out certain parameter sets. Further, we show that, when $m=3$ or $4$ and 
$v > m$,  a $(v,m,k,\lambda)$-SEDF does not exist. 
\end{abstract}

\section{Introduction}

Motivated by applications to \emph{algebraic manipulation detection codes} (or AMD 
codes) \cite{CDFPW,CFP,CPX},  Paterson and Stinson introduced {\em strong external 
difference families} (or SEDFs)  in \cite{PS}. SEDFs are closely related to but stronger than 
{\it difference systems of sets} \cite{Lev} and {\it external difference families}  (or EDFs) 
\cite{OKSS}. In \cite{PS}, it was noted that optimal AMD codes can be obtained from 
EDFs, whereas  optimal strong AMD codes can be obtained from SEDFs.  See \cite{PS} 
for a discussion of these and related structures and how they relate to AMD codes.
In this paper, we focus on SEDFs, the existence of which is an interesting mathematical 
problem in its own right, independent of any applications to AMD codes. It is important to 
note that  these bear no known relation to ``strong difference families'' (SDFs) introduced 
by Buratti  \cite{Buratti} in 1999, in spite of the misfortune of inadvertently  similar 
terminology. (See \cite{BurattiPasotti} and \cite{Momi} for some related applications.)

We recall the definition of SEDFs from \cite[Defn.\ 2.5]{PS} now. Let $G$ be a finite abelian 
group of order $v$ (written multiplicatively)  with identity element $1\in G$. For parameters 
$k,\lambda, m$ such that
\begin{equation}  
\label{Epars}
  k^2 (m-1) = \lambda (v-1)
 \end{equation}
we seek (pairwise disjoint) subsets $D_1,\ldots, D_m \subseteq G$ with $|D_j|=k$ 
$(1 \leq j \leq m)$ satisfying
\begin{equation}  
\label{E1}
 \sum_{\ell \neq j} D_j D_\ell^{-1} = \lambda (G-1) 
 \end{equation}
for each $1 \le j \le m$, where the above equation holds in the group algebra $\cx[G]$ 
and $D_\ell^{-1}$ denotes $\sum_{g \in D_\ell} g^{-1}$.  For convenience, we capture 
the remaining defining conditions here:
\begin{equation}  
\label{E0}
D_1,D_2,\ldots, D_m \subset G, \qquad |D_j|=k \ \ \forall j, \qquad |G|=v
 \end{equation}
 where $G$ is a finite abelian group.  A collection $(D_1,\ldots, D_m)$ satisfying conditions 
 (\ref{E1}) and (\ref{E0}) is denoted as a \emph{$(v,m,k,\lambda)$-SEDF}.

Let $\cD$ denote the union of all the sets $D_j$: in group algebra notation,
$$ \cD = \sum_{j=1}^m D_j . $$
With this, Equation (\ref{E1}) becomes
\begin{equation}  
\label{E2}
  D_j \cD^{-1} - D_j D_j^{-1} = \lambda (G-1) .
 \end{equation}
 
Two infinite classes of SEDFs are known, when $k=1$ and $m=2$.  These were shown in \cite{PS}.
 
\begin{example}
\label{exam1}
Let $G = (\zed_{k^2+1}, +)$, 
$D_1 = \{0, 1, \dots, k-1\}$ and $D_2 = \{k, 2k, \dots, k^2\}$.
This is a $(k^2+1,2,k,1)$-SEDF.
\end{example}

\begin{example}
\label{exam2}
Let $G = (\zed_{v}, +)$ and
$D_i = \{i\}$ for $0 \leq i \leq v-1$.
This is a $(v,v,1,1)$-SEDF.
\end{example}

The following theorem is also proved in \cite{PS}.

\begin{thm}
\label{SEDF-nonexistence}
\cite{PS}
There does not exist a $(v,m,k,1)$-SEDF with $m \geq 3$ and $k > 1$.
\end{thm}
 
 In \cite{PS}, the authors ask if there are any additional parameters for which SEDFs 
 exist.  Some necessary conditions for the existence of SEDFs have recently been 
 obtained  by Huczynska and Paterson \cite{HP} using combinatorial techniques. Their 
 results include a substantially complete treatment of the case $\lambda = 2$.
 In this paper, we  apply linear characters of $G$  to both sides of the equation 
 (\ref{E2}) to rule out the existence of SEDFs in groups of prime order. We also give 
 some partial results (i.e., necessary conditions) for (abelian) groups whose order 
 $v$ is the product of two distinct primes.
 
 We now state and prove a simple numerical result, concerning parameters for SEDFs, 
 that we will use later.
 
 \begin{lem}
 \label{num.lem}
 There does not exist a $(v,m,k,\lambda)$-SEDF with $v=mk$ and $k > 1$.
 \end{lem}
 
 \begin{proof}
 From (\ref{Epars}), we have $k^2 (m-1) = \lambda (km-1)$. Clearly, $\gcd (k, km-1) =1$, 
 so it follows that $k^2 \mid \lambda$. Thus, $\lambda = t k^2$ for some positive integer 
 $t$, and hence $m-1 = t(km-1)$. This shows that $k=t=1$.
 \end{proof}

\section{Some facts about characters}

We briefly review some basic facts about characters of finite abelian groups. These can be 
found, for example, in \cite{leder}.

For a finite abelian group $G$, there are exactly $v=|G|$ distinct homomorphisms from $G$ 
to the multiplicative group of complex numbers. First consider a cyclic group $A$ of order 
$v$ written multiplicatively as  $A =\{1=x^0,x^1, \ldots, x^{v-1}\}$ with $x^v=1$. Let 
$\omega$ denote a primitive complex $v^{\rm th}$ root of unity, e.g., $\omega=e^{2\pi i/v}$. 
For $0\le a < v$, take $\chi_a : A \rightarrow \cx^*$ via
$$ \chi_a( x^b) = \omega^{ab} . $$
For $a=0$, we have the trivial character, which we will denote by $\ones$.  Every finite abelian 
group $G$ is isomorphic to a direct product of finite cyclic groups and their  characters can be 
pieced together to give $|G|$ distinct characters of $G$. If 
$\phi: G \rightarrow A_1 \times \cdots \times A_r$ is an isomorphism where $A_j$
is cyclic of order $m_j$, say $A_j = \{x_j^0,x_j^1,\ldots, x_j^{m_j-1}\}$ with $x_j^{m_j}=x_j^0$, 
and $\omega_j$ is a primitive complex $m_j^{\rm th}$ root of unity,  we may directly 
associate\footnote{We note that this choice, while simple, is not canonical.}  a character 
$\chi_a$ of $G$ to each element. For $\phi(a)= (x_1^{a_1},\ldots, x_r^{a_r})$ and $\phi(b)= 
(x_1^{b_1},\ldots, x_r^{b_r})$, define
$$ \chi_a(b) = \omega_1^{a_1b_1} \cdots \omega_r^{a_rb_r} $$
which is readily verified to satisfy $\chi_a(gh)=\chi_a(g)\chi_a(h)$. 

When $G$ is abelian, the product of characters $(\chi \psi)(g)=\chi(g)\psi(g)$ is again a
character and, provided $G$ is finite, this gives us a group of characters isomorphic to 
$G$. So we can label the $|G|$ distinct characters $\{ \chi_a \mid a\in G\}$.  With the 
isomorphism $a\mapsto \chi_a$ between $G$ and its group of characters chosen above,
these characters  satisfy the following ``orthogonality'' relations:  $\chi_a(g) = \chi_g(a)$ and 
$$ \sum_{g\in G} \overline{\chi(g)} \psi(g) = \begin{cases} v  
&\mbox{if } \chi = \psi \\  0 & \mbox{otherwise. } 
\end{cases} $$
So, for any non-principal character (i.e., $\chi \neq \ones$), we have $\sum_g \chi(g) = 0$. 
Likewise $\sum_{\chi} \chi(g) =0$ unless $g=1$ in $G$, in which case the sum equals $v$.

Each $\chi$ extends to an algebra homomorphism from the group algebra $\cx[G]$ to $\cx$. 
Since the ``Fourier  matrix'' whose columns are the $v$ characters is invertible, this provides 
us a bijection from $\cx[G]$ to the vector space $\cx^v$. Under this bijection, we have the 
group algebra element $0=\sum_a 0 \cdot a$ mapping to the vector $\mathbf{0}$,  the group 
identity $1\in G$ mapping to the all ones vector in  $\cx^v$, and the element 
$$ G = \sum_{a\in G} a $$
mapping to the vector $(v,0,\ldots,0)$.  So $G-1$ maps to $(v-1,-1,\ldots,-1)$.

Since $G$ is finite, every value $\chi(g)$ lies on the unit circle in $\cx$.  Clearly, 
$\chi(g^{-1})=\overline{\chi(g)}$. For $S\subseteq G$, we abbreviate the group algebra 
element $\sum_{a\in S} a$ to simply ``$S$''. We then easily see that, for 
$S^{-1} := \sum_{a\in S} a^{-1}$, we have
$$ \chi( S^{-1} ) = \overline{ \chi(S) }. $$
\bigskip

We remark that most of our discussion below never relies on the particular structure of the 
group $G$. However, we do assume that $G$ is an abelian group of order $v$, not 
necessarily cyclic.

\section{Applying characters to prove results about SEDFs}

If we apply the trivial character to Equation (\ref{E2}), we obtain Equation (\ref{Epars}). 
Suppose, on the other hand, that $\chi$ is  a non-principal character. Applying $\chi$ to 
Equation (\ref{E2}), we obtain 
\begin{equation}  
\label{Echi1}
  \chi(D_j) \overline{ \chi( \cD) }  - \chi(D_j) \overline{ \chi( D_j) }  =  - \lambda  .
 \end{equation}
Since the right-hand side is non-zero, we immediately have $ \chi(D_j) \neq 0 $ for all $j$.  
Here are some more basic observations: 
 
\begin{lem}  \label{Lbasics}
\begin{itemize}
\item[(a)] For any character $\chi$ of $G$, 
\begin{equation}  
\label{Esumchi}
  \chi(\cD) = \sum_{j=1}^m \chi(D_j) 
 \end{equation} 
\item[(b)]  For any character $\chi$ of $G$,
\begin{equation}  
\label{FACTchi}
  \chi(D_j) \neq 0 \qquad \forall \ 1\le j \le m, \qquad \forall \ \chi \neq \ones.  
 \end{equation} 
\item[(c)]  If $\chi$ is a non-principal character such that $\chi( \cD)=0$, then, for all $1\le j \le m$,
$$ \chi(D_j) \overline{ \chi( D_j) }  =   \lambda  .$$
\item[(d)]  $\ones(\cD) = |\cD|=mk$ and, except when $\cD=G$ or $\cD=\emptyset$, there is at 
least one non-principal  character $\chi$ satisfying $\chi(\cD)\neq 0$. 
\end{itemize}
\end{lem}

\begin{proof}
Part (a) follows from linearity. Parts (b) and (c) follow immediately from equation (\ref{Echi1}).
For part (d), we make use of the fact that the Fourier matrix is invertible, which implies that
$\cx[G]$ is in bijective correspondence with $\cx^v$. For each
complex number $z$, we already 
know one preimage of $(z,0,\ldots,0)$; so if $\mathbf{g}\in \cx[G]$ with $\chi(\cD)=(z,0,\ldots,0)$, 
then  $\mathbf{g} = \frac{z}{v} G$.
\end{proof}

We now conjugate Equation (\ref{Echi1}) and obtain
\begin{equation}  
\label{Echi2}
  \chi(\cD)  = \frac{   | \chi(D_j) |^2 - \lambda }{  \overline{ \chi(D_j) } } \qquad \forall j ~ .  
 \end{equation} 
 
Note that, when $k=1$, the assumption that $\cD$ is a set (and not a multiset) forces 
$\cD=G$ and we have $\chi(\cD)=0$ for every non-trivial character $\chi$ of $G$. So 
our analysis does not apply in the case  $k=1$.

\begin{lem} 
\label{Lparallelchi}
Let $\chi \neq \ones$.  If  $\chi(\cD)\neq 0$, then there exist nonzero real numbers 
$\alpha_1,\ldots,\alpha_m$ such that $\chi(D_j) = \alpha_j \chi(\cD)$ for $j=1,\ldots, m$.
\end{lem}

\begin{proof} We know 
\[\overline{\chi(D_j)}^{-1} = \frac{\chi(D_j)}{|\chi(D_j)|^2},\] so 
we obtain 
\[ \chi(\cD) = \left(  \frac{   | \chi(D_j) |^2 - \lambda }{ |\chi(D_j)|^2  } \right) \chi(D_j)
\]
from equation (\ref{Echi2}).
$\chi(\cD)$ and $\chi(D_j)$ are both nonzero, so 
\[ \alpha_j = \frac{ |\chi(D_j)|^2  }{   | \chi(D_j) |^2 - \lambda } \] 
is nonzero. Further, $\alpha_j$ is clearly a real number.
\end{proof}

So let's  fix $\chi \neq \ones$ and, assuming $\chi(\cD)\neq 0$, denote $\chi(\cD)= X$ 
and $\chi(D_j)=x_j=\alpha_j X$. We obtain a system of quadratic equations from 
(\ref{Echi1}) which gives us restrictions on the values $\alpha_j$. For $m=2$, this 
simply tells us $\alpha_1 \alpha_2 = -\lambda/X\bar{X}$ and rules nothing out. However, 
for $m \geq 3$, we obtain some useful information. We first consider the case $m=3$.
 
 \begin{thm}
 \label{m=3}
 There does not exist a $(v,3,k,\lambda)$-SEDF for any $v> 3$.
 \end{thm}
 
 \begin{proof}
First, since $v > 3$, it follows from Lemma \ref{num.lem} that we cannot have $\cD=G$. 
Hence, from Lemma \ref{Lbasics}(d), there is a non-principal character  $\chi$ such that 
$\chi( \cD)\neq 0$.  Now, from (\ref{Echi1}), we have
 $$ \begin{array}{ccccccc}
   x_1 \bar{x}_2 &   +   &   x_1 \bar{x}_3 &    &       &   =   &   -\lambda \\
   x_2 \bar{x}_1 &   &   &    +   &   x_2 \bar{x}_3 &     =   &   -\lambda \\
     &    &   x_3 \bar{x}_1 &   +   &   x_3 \bar{x}_2 &      =   &   -\lambda .
     \end{array} $$
We conjugate the second equation and subtract it from the first to obtain
$ x_1 \bar{x}_3 = \bar{x}_2 x_3$. 
In the same manner, we find  $x_1 \bar{x}_2 = x_2 \bar{x}_3 = x_3 \bar{x}_1$. 
Hence, $\alpha_1 \alpha_2 = \alpha_2 \alpha_3 = \alpha_3 \alpha_1$.
This forces $\alpha_1=\alpha_2=\alpha_3$ since all $\alpha_i$'s are nonzero.  
But now we have a contradiction: our first equation becomes
$$ \alpha_1^2 X\bar{X} + \alpha_1^2 X\bar{X} = - \lambda $$
with the lefthand side nonnegative and the righthand side negative.  This shows 
$m=3$ cannot occur.
\end{proof}

For larger values of $m$, (\ref{Echi1}) gives us $m$ equations
\begin{equation}
\label{Exlam}
 \sum_{\ell \neq j} x_j \bar{x}_\ell  = -\lambda
\end{equation}
($j=1,\ldots,m$). For any distinct indices $r$ and $s$, we find
$$ x_r \bar{x}_s + \sum_{\ell \neq r,s} x_r \bar{x}_\ell = -\lambda $$
and, after conjugation,
$$ x_r \bar{x}_s + \sum_{\ell \neq r,s} \bar{x}_s x_\ell = -\lambda $$
so that
\begin{equation}
\label{allbuttwo}  \sum_{\ell \neq r,s} x_r \bar{x}_\ell  =  \sum_{\ell \neq r,s} \bar{x}_s x_\ell . 
\end{equation}
Therefore, we have 
$$  \sum_{\ell \neq r,s} \alpha_r \alpha_\ell X\bar{X}  =  \sum_{\ell \neq r,s}   \alpha_s \alpha_\ell X\bar{X}    . $$
Recall we are assuming $X\neq 0$. So we obtain
\begin{equation}
\label{Erst}
 \alpha_r  \left( \sum_{\ell \neq r,s} \alpha_\ell \right)  =  \alpha_s \left( \sum_{\ell \neq r,s}   \alpha_\ell \right)   . 
 \end{equation}
 
\begin{lem}
\label{L<=2}
For $m>3$, there can be only two possible values for $\alpha_j$ ($1\le j\le m$) when $X\neq 0$.
\end{lem}

\begin{proof} Consider three distinct indices $r,s,t$. If $\alpha_r\neq \alpha_s$, then 
Equation (\ref{Erst}) gives
$$ \alpha_t = - \sum_{\ell \neq r,s,t} \alpha_\ell . $$
Likewise, if $\alpha_r \neq \alpha_t$, then 
$$ \alpha_s = - \sum_{\ell \neq r,s,t} \alpha_\ell ,$$
i.e., $\alpha_s=\alpha_t$.  \end{proof}

\subsection{Highly uneven---or even---splits cannot occur}

We first introduce some notation. We assume that $m > 3$, and
$$ \{ \alpha_1,\ldots, \alpha_m\} \subseteq \{ \alpha, \beta \}$$
with $\alpha_j = \alpha$ for exactly $A$ values of $j$ and $\alpha_j = \beta$ for 
exactly $B$ values of $j$. Without loss of generality, we assume $0\le A\le B$, 
and we know that $A+B=m$. The ordered pair $(A,B)$ will be called the {\em split} 
of $(\alpha_1,\ldots, \alpha_m)$.

\begin{lem}
\label{Lsplit}
Let $m > 3$ and $\chi \neq \ones$ with $X\neq 0$. Then the split of 
$(\alpha_1,\ldots, \alpha_m)$ cannot be $(0,m)$, $(1,m-1)$ or $(m/2,m/2)$.
\end{lem}

\begin{proof}  We have
\begin{equation}
\label{EL2}
\alpha A + \beta B  = 1
\end{equation}
since $\sum_j x_j = X$,  and
\begin{equation}\label{EL3}
\alpha_j \left( \sum_{\ell \neq j} \alpha_\ell \right) = -\frac{\lambda}{X\bar{X}}
\end{equation}
from Equation (\ref{Exlam}). 

If $A=0$ (i.e., all the $\alpha_j$'s are equal), then $\beta=1/m$ by (\ref{EL2}). But this 
gives a positive value for the left-hand side of (\ref{EL3}) while the right-hand side is 
negative, a contradiction. 

Next, if $A=1$ and $B=m-1$, we employ (\ref{Erst}) to find $\alpha ((m-2)\beta) = 
\beta((m-2)\beta)$. Since $\beta\neq 0$ by Lemma \ref{Lparallelchi}, this gives 
$\alpha=\beta$, which is the case we just discussed.

Finally, if $m$ is even and $A=B$, then (\ref{Erst}) with $\alpha_r=\alpha$ and 
$\alpha_s=\beta$ gives $\alpha+\beta=0$. But this is impossible since, if $A=B=m/2$, 
we must have $\alpha + \beta = 2/m$ from (\ref{EL2}) . 
\end{proof}

For small values of $m$, the three cases handled in Lemma \ref{Lsplit} cover all or 
most of the possible splits.

\begin{thm}  
\label{m=4} 
 There does not exist a $(v,4,k,\lambda)$-SEDF  for any $v> 4$.
 \end{thm}

\begin{proof}
 First, since $v > 4$, it follows from Lemma \ref{num.lem} that we cannot have 
 $\cD=G$.  Hence, from Lemma \ref{Lbasics}(d), there is a non-principal character 
 $\chi$ such that $\chi( \cD)\neq 0$. The result then follows from Lemma \ref{Lsplit}.
\end{proof}

The following also follows from Lemma \ref{Lsplit}.

\begin{thm} 
\label{splits} Suppose there is a $(v,m,k,\lambda)$-SEDF with $k > 1$.
\begin{itemize}
\item[(a)]
If $m=5$, the split for any $\chi\neq \ones$ with $\chi(\cD)\neq 0$ must be 
$(A,B)=(2,3)$. 
\item[(b)]
If $m=6$, the split for any $\chi\neq \ones$ with $\chi(\cD)\neq 0$ must be 
$(A,B)=(2,4)$.  $\Box$
\end{itemize}
\end{thm}
 
\begin{example}  
We analyze the case $m=5$. Restrict to some $\chi\neq \ones$  with 
$X:=\chi(\cD) \neq 0$.  From Theorem \ref{splits}, without loss of generality, 
there exist real numbers $\alpha$ and $\beta$ with
 $$ \alpha_1 = \alpha_2 = \alpha \ \ \neq \ \ \beta  = \alpha_3 =  \alpha_4 = \alpha_5 . $$
Equation (\ref{Erst}) with $r=1$, $s=3$ gives $\alpha+2\beta=0$. Combining this 
with (\ref{EL2}), we find
 $$ \alpha = 2, \quad \beta=-1 ~. $$
 So 
 $$ \chi( D_r + D_s + D_t ) = 0 $$
whenever $r\in \{1,2\}$ and $s,t\in \{3,4,5\}$.  This gives us  $ 2 |X|^2 = \lambda $
using Equation (\ref{Echi2}), which seems perfectly valid at this point. We have six 
different proper subsets of $G$ over which the character $\chi$ sums to zero.  This 
happens (for some choices of $\chi$) when the subset is a coset of a proper 
subgroup of $G$, so perhaps this is possible.  
\end{example}

\subsection{The case of prime $v$}

In this section, we make use of some results of Lam and Leung \cite{LL}. We first 
summarize the material we need from their paper in a single theorem.

\newcommand{\cE}{\mathcal{E}}

\begin{thm}\cite{LL} 
\label{TLL}
Let $G$ be a finite cyclic group of order $v$, 
where $v$ has  prime power factorization $v=p_1^{r_1}\cdots p_s^{r_s}$
and let $\chi$ be a non-principal character of $G$. Let $\cE \in \ints[G]$, say
$$ \cE = \sum_{g\in G} n_g g . $$
\begin{itemize}
\item[(a)]
If   $\chi( \cE )=0$,   then there exist nonnegative integers $k_1,\ldots, k_s$ such that
$$ \sum_{g \in G} n_g = k_1 p_1 + \cdots + k_s p_s. $$
\item[(b)] 
If $v$ is prime and $\cE \subseteq G$ with $\chi(\cE)=0$, then $\cE=\emptyset$ 
or $\cE=G$.
\item[(c)]
If $v=pq$, where $p$ and $q$ are distinct primes, and $\cE \subseteq G$ with 
$\chi(\cE)=0$,  then there exists a nonnegative integer $k$ such that $|\cE|=kp$ or 
$|\cE| = kq$.  Further, $\cE$ is expressible as a disjoint union of cosets of some 
proper subgroup $H$ of $G$ (where $|H|=p$ or  $|H|=q$).
\end{itemize}
\end{thm}

The last statement (part (c)) is not directly given in the paper of Lam and Leung. 
But it follows immediately from their results.  They prove in \cite{LL} that, if 
$\cE=\sum_{g\in G} n_g g$ is a multiset (i.e., all  $n_g\ge 0$),  then $\sum_{g} n_g$
is expressible as $kp+k'q$ for non-negative integers $k$ and $k'$.  But, by the 
Chinese Remainder Theorem, we cannot have both $k$ and $k'$  nonzero when 
$\cE$ is simply a subset of $G$ (because any coset of a subgroup  of order $p$ 
intersects any coset of a subgroup of order $q$).

Here is the main theorem of this section, which completely handles the case 
where the group $G$ has prime order.

\begin{thm} If $v$ is prime, $k>1$ and $m>2$, then there does not exist a
$(v,m,k,\lambda)$-SEDF.
\end{thm}

\begin{proof} We can assume $m \geq 5$ in view of Theorems \ref{m=3} and 
\ref{m=4}. Select $\chi\neq \ones$ with $\chi(\cD)\neq 0$. By Lemma \ref{Lsplit}, 
we have two distinct real numbers $\alpha,\beta$ such that $\alpha_j
\in \{\alpha,\beta\}$ for $1\le j \le m$, with each value occurring at least twice. 
Choose $r$ and $s$ with  $\alpha_r=\alpha$ and $\alpha_s = \beta$. Then 
Equation (\ref{Erst}) forces 
\begin{equation}
\label{chi=0}
 \chi \left( \sum_{\ell \neq r,s} D_\ell \right) = 0. 
 \end{equation}
However, the set \[\bigcup_{\ell \neq r,s} D_\ell \] 
has cardinality $k(m-2)$ which lies strictly between $0$ and $v$. This
contradicts Theorem \ref{TLL}(b).  
\end{proof}

\subsection{The case $v = pq$, where $p$ and $q$ are distinct primes}

Suppose $v = pq$, where $p$ qnd $q$ are distinct primes. We assume $k >1$ and 
$m > 2$.  We have the following:

\begin{eqnarray}
k^2 (m-1) & = & \lambda (v-1) \label{14}\\
v &=& pq\label{15}\\
k(m-2) &\equiv & 0 \bmod p\label{16}\\
km & < & v\label{17}.
\end{eqnarray}
Note that (\ref{14}) is just (\ref{Epars}). The congruence
(\ref{16}) follows from (\ref{chi=0}) and Theorem \ref{TLL}(c), where
without loss of generality we can assume that (\ref{16}) holds for  the prime 
divisor $p$  (if not, we can interchange $p$ and $q$). 
Finally, (\ref{17}) is just Lemma \ref{SEDF-nonexistence}.

 From (\ref{16}), there are two possible cases to consider:
\begin{description}
\item[case 1:]
$p \mid k$, or 
\item[case 2:] $p \mid (m-2)$. 
\end{description}
We now derive some necessary conditions that must hold if we are in case 1. 
Suppose $p \mid k$;  then $k = sp$, where $s$ is a positive integer.
From (\ref{14}), we have $p^2 \mid \lambda$, so let $\lambda = tp^2$, where 
 $t$ is a positive integer. Then (\ref{14}) becomes
\[ s^2 (m-1) = t(v-1).\]
Now, $spm = km < v = pq$, so $sm < q$.
Further,
\[ t(v-1) = s^2 (m-1) < s^2 m \leq qs,\]
so we have
\[ s > \frac{t(v-1)}{q}.\]
Then,
\[k = sp >  \frac{pt(v-1)}{q} \geq \frac{p(v-1)}{q}.\]
Now, if $k > v/5$, then $m < 5$, which is impossible by
Theorems \ref{m=3} and \ref{m=4}.
Therefore,
\begin{equation}
\label{ratio} \frac{p}{q} \leq \frac{v}{5(v-1)} = \frac{1}{5} + \frac{1}{5(v-1)}.
\end{equation}

This inequality (\ref{ratio}) can sometimes be used to rule out certain parameter 
situations. However, there certainly are parameter sets that satisfy all the 
necessary conditions given above.

\begin{example}
Suppose $p = 7$, $q = 31$, $v = pq = 217$, $m = 9$, $k=9$ and 
$\lambda = 3$.  Here $m \equiv 2 \bmod p$, $k^2 (m-1) = 648 = \lambda (v-1)$ 
and $km = 81 < v$. So the above arguments do not rule out the existence of a 
$(217,9,9,3)$-SEDF.
\end{example}

\section{Conclusion}

We posted this paper on ArXiv on 20 October, 2016. 
Shortly after that, there followed a flurry of work on the topic of SEDFs by a 
variety of researchers \cite{BJWZ,HP,JL,WYF,WYFF}; all of these preprints 
have also appeared on ArXiv. These papers contain a number of interesting 
results, including constructions as well as new nonexistence results. There 
are now a number of known constructions of SEDFs for the case $m=2$,
as well a single nontrivial example with $m > 2$, namely, a $(243,11,22,20)$-SEDF
that is presented in \cite{JL,WYF}. We encourage readers who are interested in 
SEDFs to read the above-referenced preprints.

\section*{Acknowledgements}

WJM thanks the Cheriton School of Computer Science (University of Waterloo)
for hosting him while the initial work on this project was done. 
He is also grateful to Jim Davis and the University of Richmond for 
hosting a helpful workshop related to this subject.

Thanks to the referee, and to Jonathan Jedwab and Shuxing Li, for helpful 
comments on earlier versions of this paper.

Research of DRS  is supported by  NSERC discovery grant RGPIN-03882.


\end{document}